\documentclass[reqno, 12pt]{amsart}
\pdfoutput=1
\makeatletter
\let\origsection=\section \def\section{\@ifstar{\origsection*}{\mysection}} 
\def\mysection{\@startsection{section}{1}\z@{.7\linespacing\@plus\linespacing}{.5\linespacing}{\normalfont\scshape\centering\S}}
\makeatother        

\usepackage{amsmath,amssymb,amsthm}
\usepackage{mathrsfs}
\usepackage{mathabx}\changenotsign
\usepackage{dsfont}
 
\usepackage{xcolor}
\usepackage[backref]{hyperref}
\hypersetup{
    colorlinks,
    linkcolor={red!60!black},
    citecolor={green!60!black},
    urlcolor={blue!60!black}
}

\usepackage[open,openlevel=2,atend]{bookmark}

\usepackage{doi}

\usepackage[abbrev,msc-links,backrefs]{amsrefs}

\renewcommand{\PrintDOI}[1]{\doi{#1}}

\usepackage[T1]{fontenc}
\usepackage{lmodern}
\usepackage[babel]{microtype}

\usepackage[english]{babel}

\usepackage{geometry}
\usepackage{enumitem}
\def\rmlabel{\upshape({\itshape \roman*\,})}

\def\alabel{\upshape({\itshape \alph*\,})}

\let\polishlcross=\l
\def\l{\ifmmode\ell\else\polishlcross\fi}

\def\qqand{\qquad\text{and}\qquad}

\let\emptyset=\varnothing
\let\setminus=\smallsetminus

\makeatletter
\def\moverlay{\mathpalette\mov@rlay}
\def\mov@rlay#1#2{\leavevmode\vtop{   \baselineskip\z@skip \lineskiplimit-\maxdimen
   \ialign{\hfil$\m@th#1##$\hfil\cr#2\crcr}}}
\newcommand{\charfusion}[3][\mathord]{
    #1{\ifx#1\mathop\vphantom{#2}\fi
        \mathpalette\mov@rlay{#2\cr#3}
      }
    \ifx#1\mathop\expandafter\displaylimits\fi}
\makeatother

\newcommand{\dcup}{\charfusion[\mathbin]{\cup}{\cdot}}

\let\eps=\varepsilon
\let\epsilon=\varepsilon
\let\theta=\vartheta
\let\rho=\varrho
\let\phi=\varphi

\def\eu{\textrm{e}}

\def\ccC{{\mathscr{C}}}
\def\ccS{{\mathscr{S}}}

\newtheorem{thm}{Theorem}[section]
\newtheorem{cor}[thm]{Corollary}

\newtheorem{fact}[thm]{Fact}

\newtheorem{claim}[thm]{Claim}

\def\cB{\mathcal{B}}
\def\cC{\mathcal{C}}

\def\cH{\mathcal{H}}
\def\cS{\mathcal{S}}

\def\bcS{\boldsymbol{\cS}}

\def\NN{\mathds{N}}

\def\eps{\epsilon}

\newcommand{\pr}{\mathds{P}}
\newcommand{\ex}{\mathds{E}}
\newcommand{\kap}{\textsl{AP}_k}
\newcommand{\Wap}{\textsl{AP}_W}

\DeclareMathOperator{\girth}{girth}
\DeclareMathOperator{\vdW}{vdW}

\begin{document}
\title[Ramsey numbers involving large girth graphs and hypergraphs]{Ramsey-type numbers involving graphs\\ and hypergraphs with large girth}

\author[Hi\d{\^e}p H\`an]{Hi\d{\^e}p H\`an}
\address{Instituto de Matem\'aticas, Pontificia Universidad Cat\'olica de Valpara\'\i{}so, Valpara\'\i{}so, Chile}
\email{han.hiep@googlemail.com}

\author[Troy Retter]{Troy Retter}

\author[Vojt\v{e}ch R\"{o}dl]{Vojt\v{e}ch R\"{o}dl}
\address{Department of Mathematics and Computer Science, 
Emory University, Atlanta, USA}
\email{\{\,tretter\,|\,rodl\,\}@mathcs.emory.edu}

\author[Mathias Schacht]{Mathias Schacht}
\address{Fachbereich Mathematik, Universit\"at Hamburg, Hamburg, Germany}
\email{schacht@math.uni-hamburg.de}
\thanks{H.~H\`an was partly supported by FAPESP (2010/16526-3 and 2013/11353-1).
V.~R\"odl was supported by NSF grant DMS 1301698.
M.~Schacht was supported through the Heisenberg-Programme of the DFG\@.}


\begin{abstract} A question of~Erd{\H o}s asks if for every pair of positive integers~$r$ and~$k$, there exists a graph~$H$ having~$\girth(H)=k$ and the property that every~$r$-colouring of the edges of~$H$ yields a monochromatic cycle~$C_k$. The existence of such graphs was confirmed by the third author and~Ruci\'nski.

We consider the related numerical problem of determining the smallest such graph with this property. We show that for integers~$r$ and~$k$, there exists a graph~$H$ on~$R^{10k^2} k^{15k^3}$ vertices 
(where~$R = R(C_k;r)$ is the~$r$-colour Ramsey number for the cycle~$C_k$) having $\girth(H)=k$ and the Ramsey property 
that every~$r$-colouring of~$E(H)$ yields a monochromatic~$C_k$. Two related numerical problems regarding arithmetic progressions in sets and cliques in graphs are also considered.

\end{abstract}

\maketitle

\section{Introduction}\label{sec:intro}
For an integer~$r\geq 2$ and graphs~$H$ and~$F$, we write~$H \rightarrow (F)_{r}$ if every~$r$-colouring of the edges of~$H$ yields a monochromatic copy of~$F$. If~$H \rightarrow (F)_{r}$, we say that~$H$ is 
\emph{Ramsey for}~$F$ for~$r$ colours. It follows from Ramsey's theorem that for every graph~$F$ and for every positive integer~$r$, there exists a graph~$H$ such that~$H \rightarrow (F)_{r}$. 
We consider three Ramsey-type problems that pertain to cycles in graphs and hypergraphs.

\subsection{Cycles in Graphs}

Our first result relates to a problem suggested by Erd{\H o}s (see, e.g.,~\cite{E}), which asks if for every pair of positive integers~$r$ and~$k$, there exists 
a graph~$H$ having~$\girth(H)$ at least $k$ and the Ramsey property~$H \rightarrow (C_k)_r$. The existence of such graphs was first established in~\cite{RR}. 
We address the associated numerical problem.

\begin{thm}\label{thm_cycles}
Let~$R = R(C_k;r)$ be the Ramsey number that denotes the least integer~$m$ such that~$K_m \rightarrow(C_k)_r$. Then for all integers~$k \geq 4$ and~$r \geq 2$, there exists a graph~$H$ on~$|V(H)|=k^{15 k^3}R^{10k^2}$ vertices satisfying $\girth(H)=k$ and $H \rightarrow (C_k)_r$.
\end{thm}

The exponential dependency of~$|V(H)|$ on~$k$ in Theorem~\ref{thm_cycles} is unavoidable. This follows from the observation that a minimal graph~$H$ with the desired properties must have minimum degree greater than~$r$ and girth at least~$k$. Note that the $r$-colour Ramsey number~$R(C_k;r)$ is polynomial in~$r$ for fixed even~$k$, while for fixed odd~$k$ it satisfies the exponential relation~$c_1^r\leq R(C_k;r)\leq c_2^{r\log r}$ for some positive constants~$c_1$ and~$c_2$\footnote{We denote by $\log$ the binary logarithm and  by $\ln$ the natural logarithm.}(see, e.g.,~\cite{EG}). This 
leads to the following corollary, which shows that the additional girth requirement on~$H$ leads to a Ramsey graph of order comparable to $R(C_k,r)$.
 
\begin{cor}
For every integer~$k\geq 3$, there exist constants~$K_{\text{e}}$ and~$K_{\text{o}}$ such that for every integer~$r \geq 2$, there exists a graph~$H$ such that~$\girth(H)=k$ and~$H\to (C_k)_r$, which satisfies~$|V(H)| \leq r^{K_{\text{e}}}$ if~$k$ is even and~$|V(H)| \leq K_{\text{o}}^{r\log r}$ if~$k$ is odd.
\end{cor}

In Section~\ref{sec_con_rem}, we will further expand upon Theorem~\ref{thm_cycles}. 
In particular, 
we prove a lower bound and give a simpler proof for the cases~$k=4$ and~$k=6$.

\subsection{Arithmetic Progressions}\label{sec:ivdW}

For a subset~$S \subseteq \NN$ and integers $k\geq 3$ and $r\geq 2$, we write~$S \rightarrow (\kap)_r$ to signify that
every~$r$-colouring of the integers in~$S$ yields a monochromatic arithmetic progression of length~$k$. Van der Waerden's theorem shows for all integers~$k\geq 3$ and $r\geq 2$ that
there is some integer~$N$ 
such that~$[N] \rightarrow (\kap)_r$, where we denote by~$[N]$ the set of the first $N$ positive integers $\{1,2,\dots,N\}$. Several refinements of this well known theorem have been considered. One generalisation, suggested by Erd{\H o}s~\cite{E2}, asks if for all~$k\geq 3$ and $r\geq 2$, there exists an~$\textsl{AP}_{k+1}$-free set~$S \subseteq \NN$ that has the property~$S \rightarrow (\kap)_r$, where a set is~$\textsl{AP}_{k+1}$-free if it does not contain an arithmetic progression of length~$k+1$. This was answered independently by Spencer~\cite{S} and by Ne\v set\v ril and R\"odl~\cite{NR2}. Moreover, Graham and Ne\v set\v ril~\cite{GN} showed that there exist arbitrarily large~$\textsl{AP}_{k+1}$-free sets~$S$ that have the property~$S \rightarrow (\kap)_r$ and are minimal in the sense that, for every~$s \in S $, the subset~$S' = S \setminus \{s\}$ fails have the van der Waerden property, i.e.,~$S' \nrightarrow (\kap)_r$.

Furthermore, one may want to restrict the structure of the arithmetic progressions of length~$k$ in a set~$S \subseteq \NN$, but keep the van der Waerden property. That is, consider the \emph{system of copies} of arithmetic progression of length~$k$ in~$S$, which is the~$k$-uniform hypergraph $(S,\binom{S}{\kap})$
on the vertex set~$S$ with edge set~$\binom{S}{\kap}$
consisting of the~$k$ element subsets of~$S$ that form arithmetic progressions of length~$k$. For a simpler notation, it will be convenient to 
identify this hypergraph just by its edge set.
Moreover, we denote its chromatic number simply by~$\chi\binom{S}{\kap}$ instead of $\chi(\binom{S}{\kap})$.
Similarly, we suppress the outer pair of parentheses for other numerical hypergraph parameters as well.

Observe that~$S \rightarrow (\kap)_r$ if and only if the chromatic number satisfies~$\chi \binom{S}{\kap} > r$. Hence, van der Waerden's theorem establishes that for fixed~$k$, the chromatic number~$\chi\binom{[N]}{\kap}\to\infty$ as~$N$ tends to infinity.
In view of the result of  Erd{\H o}s and Hajnal~\cite{EH}, which establishes the existence of hypergraphs having both large chromatic number and large girth, it is naturally to ask the following. Does for all~$k$, $g \geq 3$, and  $r \geq 2$ there exist a set~$S \subseteq \NN$ so that the hypergraph~$\binom{S}{\kap}$  satisfies both the properties
\begin{enumerate}[label=\upshape({P\arabic*})]
\item\label{P1} $\chi \binom{S}{\kap}  > r$,
\item\label{P2} $\text{girth} \binom{S}{\kap}  \geq g\,?$
\end{enumerate}
As usual we say a~$k$-uniform hypergraph has girth at least~$g$ if, for any integer $h$ with $2 \leq h < g$, any subset of~$h$ edges span at least~$(k-1)h + 1$ vertices. 
In particular, $\text{girth} \binom{S}{\kap}\geq 3$ implies that no two arithmetic progressions can intersect in more than one point, which implies 
that~$S$ is $\textsl{AP}_{k+1}$-free.
The existence of sets $S\subseteq \NN$ satisfying properties~\ref{P1} and~\ref{P2} was established in~\cite{R} (see also~\cite{RR})
and our next result gives an  upper bound for the size of the smallest such set $S$.

\begin{thm}\label{thm_ap}
Let $W = \vdW(k,r)$ denote the least integer~$N$ such that~$[N] \rightarrow (\kap)_r$. 
Then for all integers~$k$, $g\geq 3$, and $r\geq 2$, there exists a set~$S \subseteq \NN$ such that
\[
	\chi {S \choose \kap} > r\,,  \qquad  
	\girth{S \choose \kap} \geq g\,, \qqand 
	|S| \leq k^{40k^2(k+g)}W^{12k(k+g)}\,. 
\]
\end{thm}

To illustrate the result, consider the special case~$k=3$ for fixed $g\geq 3$. A result of Sanders~\cite{Sanders} (see also~\cite{Bl}) implies 
that~$\vdW(3;r)\leq \exp\left(r^{1+o(1)}\right)$, where the error term~$o(1)\to 0$ as~$r\to \infty$.
Consequently, our result yields the existence of a set~$S$ of size at most 
$\exp\left(r^{1+o(1)}\right)$ such that the properties $S\to(\textsl{AP}_3)_r$ and $\binom{S}{\textsl{AP}_3}\geq g$ both hold. 
In other words, as in Theorem~\ref{thm_cycles} 
the added girth condition does not essentially increase the best known upper bound in this case.

\subsection{Cliques in Graphs}

Another well known problem of Erd{\H o}s and Hajnal~\cite{EH-folk} asked if, for every pair of positive integers~$k$ and~$r$, there exists a~$K_{k+1}$-free graph~$H$ such that~$H \rightarrow (K_k)_r$. The case~$r = 2$ was confirmed by Folkman~\cite{Folk}, and
the general case~$r > 2$ was resolved by Ne\v set\v ril and R\"odl~\cite{NR-Folk}. Subsequently, Erd{\H o}s~\cite{E} asked for a strengthened form of this result, namely the existence of a graph~$H$ with~$H \rightarrow (K_k)_r$ in which no two copies of~$K_k$ share more than one edge, which was established in~\cite{NR1} (see also~\cite{NR} for a generalisation from cliques $K_k$ to arbitrary graphs).

As in the context of van der Waerden's theorem in Section~\ref{sec:ivdW}, we may consider the structure of the cliques in~$H$ in more detail, i.e., 
we consider the \emph{system of copies} of~$K_k$ in~$H$, which is the~$\binom{k}{2}$-uniform hypergraph $(E(H),\binom{H}{K_k})$ having vertex set~$E(H)$ and hyperedges corresponding to the edge sets of copies of~$K_k$ in~$H$. As above we identify this hypergraph by its edge set $\binom{H}{K_k}$ and denote by  
$\chi\binom{H}{K_k}$ and $\textrm{girth}\binom{H}{K_k}$ its chromatic number and its girth. Again the statement~$H \rightarrow (K_k)_{r}$ is equivalent 
to~$\chi \binom{H}{K_k} > r$ and the property that any two copies of~$K_k$ in~$H$ share at most one edge is equivalent to~$\textrm{girth}\binom{H}{K_k}\geq 3$. 
We give a new proof of the result from~\cite{NR1} that leads to a new  upper bound on the size of the smallest such~$H$. 
 
\begin{thm}\label{thm_cliques}
Let~$R=R(K_k;r)$ be the Ramsey number that denotes the least integer~$m$ such that~$K_m \rightarrow (K_k)_r$. Then for all integers~$k$, $g\geq 3$, and $r\geq 2$, there exists a graph~$H$ such that 
\[
	\chi {{H} \choose {K_k} } > r\,,  \qquad 
	\girth{{H} \choose {K_k} } \geq g\,, \qqand 
	|V(H)| \leq k^{40gk^4}R^{40gk^2}\,.
\]
\end{thm}
By reversing the dependency between~$g$ and~$|V(H)|$, we obtain the following corollary. 
\begin{cor}\label{cor_cliques}
For all integers~$k\geq 3$ and~$r\geq 2$, there exist~$c_{k,r}>0$ and $n_0$
such that, for every integer~$n \geq n_0$, 
there exists a graph~$H$ on~$n$ vertices satisfying both~$H\to (K_k)_r$ and~$\girth\binom{H}{K_k}\geq c_{k,r} \log n$.
\end{cor}
It can be shown that any graph~$H$ on~$n$ vertices satisfying~$H\to (K_k)_r$ must also satisfy~$\girth\binom{H}{K_k}=O(\log n)$, due to the
 degree condition
required by $\chi\binom{H}{K_k}>r$ and, in that sense, our result gives an optimal order of magnitude for the girth. 

\subsection{Organization}
In Section~\ref{sec_containers} we state the so-called \emph{Container Lemma}, which playes a central r\^ole 
in our proofs. The details of the proofs of Theorems~\ref{thm_cycles}, \ref{thm_ap}, and~\ref{thm_cliques} 
are given in Sections~\ref{sec_cycles},~\ref{sec_ap}, and \ref{sec_cliques} respectively. 
Section~\ref{sec_con_rem} contains some  remarks related to Theorem~\ref{thm_cycles}.

\section{Hypergraph Containers} \label{sec_containers}
The proofs of the theorems presented in Section~\ref{sec:intro} use some ideas from~\cites{NS,RRS} and
rely on random constructions combined with the \emph{Container Method}
of Balogh, Morris, and Samotij~\cite{BMS} and of Saxton and Thomason~\cite{ST}.
For the numerical aspects the container
result from~\cite{ST} seemed to be better suited and we state it below (see~Theorem~\ref{thm:cl}).

Roughly speaking, this lemma states that, if a given hypergraph~$\mathcal{H}$ satisfies some numerical
`degree conditions', then there there exists a relatively `small' family of sets of so-called `containers' (sets~$\cC$ in Theorem~\ref{thm:cl} below) that
are `almost' independent sets of vertices that cover all independent sets of~$\cH$.

We now introduce the notation necessary for the  formulation of this theorem. For a~$h$-uniform hypergraph~$\mathcal{H}$, let~$e(\mathcal{H})$ denote the number of (hyper)edges in~$\mathcal{H}$.
For a set~$J \subseteq V(\mathcal{H})$ we define the \emph{degree of~$J$} by
\[
	d(J)= |\{e\in E(\mathcal{H})\colon e \supseteq J \}|
\] 
and  for $j=1,\dots,h$ we define the \emph{maximum $j$-degree} of a vertex $v\in V(\cH)$ by

\[
	d_j(v)= \max \Big\{ d(J)\colon J \in \tbinom{V(H)}{j} \text{ and } v \in J \Big\}\,.
\] 
The average of $d_j(v)$  is denoted by
\[
	d_j =\frac{1}{|V(\mathcal{H})|} \sum_{v \in V} d_j(v)\,.
\]
Note that~$d_1(v)$ is just the degree of~$v$ in~$\cH$ and, consequently, $d_1$ denotes the 
average vertex degree of~$\cH$. With this notation at hand we state the \emph{Container Lemma}
from~\cite{ST}*{Corollary~3.6}.

\begin{thm}[Container Lemma]\label{thm:cl}
Let~$\mathcal{H}=(V,E)$ be a~$h$-uniform hypergraph and suppose~$\tau$, $\epsilon \in (0,1/2)$ satisfy
\begin{equation} \label{containhyp}
\frac{6 \cdot h! \cdot 2^{{h \choose 2}}}{d_1} \sum_{j=2}^h \frac{d_j}{2^{j-1 \choose 2}\tau^{j-1}} \leq \epsilon\,.
\end{equation}
Then for integers 
\[
	K=800h(h!)^3
	\qqand
	s=\lfloor K\log(1/\eps)\rfloor
\]
the following holds.

For every independent set $I \subseteq V$ in $\cH$, there exists 
an $s$-tuple $\bcS=(\cS_1,\dots,\cS_s)$ of subsets of~$V$ and 
a subset $\cC=\cC(\bcS)\subseteq V$ only depending on $\bcS$ such that 
\begin{enumerate}[label=\rmlabel]
\item\label{cl:i}  $\bigcup_{i\in[s]}\cS_i\subseteq I\subseteq \cC$,
\item\label{cl:ii} $e(\mathcal{C}) \leq \epsilon\cdot  e(\mathcal{H})$, and 
\item\label{cl:iii} for every $i\in[s]$ we have $|\cS_i|\leq  \tau K |V|$.
\end{enumerate}
\end{thm}

The Container Lemma stated here is an abridged version of~\cite{ST}*{Corollary~3.6}, which suffices for our purpose.
For the explicit constant 800 appearing in the constant~$K$ see the discussion following Corollary~3.6 in~\cite{ST}.

\section{Proof of Theorem~\ref{thm_cycles}} \label{sec_cycles}
In this section we prove Theorem~\ref{thm_cycles}. For that we have to show that 
there exists a graph~$H$ on at most~$k^{15 k^3}R^{10k^2}$ vertices
(where~$R=R(C_k;r)$ is the~$r$-colour Ramsey number for~$C_k$)
with~$\girth(H)=k$ and the Ramsey property~$H \rightarrow (C_k)_r$.

\begin{proof}
For integers~$k \geq 4$ and~$r \geq 2$ let $R=R(C_k;r)$ be the $r$-colour Ramsey number for~$C_k$.
We first define all constants involved in the proof. For the application of the Container Lemma we 
set 
\begin{equation}
	\label{eq:cyclesepsDtau}
	\eps = \frac{1}{r R^k}
	\qqand
	D_\tau=\frac{2^{2k}}{\eps^{1/(k-1)}}
\end{equation}
and we fix integers
\begin{equation}
	K=800k\big(k!\big)^3<30k^{3k}
	\qqand 
	s=\lfloor K\log(1/\eps)\rfloor
	<30k^{3k}\log(rR^k)\,.
\end{equation}
We set
\begin{equation}\label{eq:clDpcyc}
	D_p=10R^2r^2 s^2KD_\tau\log(10R^2r)<k^{15k}R^{10}
\end{equation}
and define 
\begin{equation}
	\label{eq:cyclesn}
	n=D_p^{k^2}<k^{15k^3} R^{10 k^2}\,.
\end{equation}
Finally, we define the following parameters appearing in the proof
\begin{equation}\label{eq:cycpt}
	\tau=D_\tau n^{-\frac{k-2}{k-1}}
	\qqand
	p=D_p n^{-\frac{k-2}{k-1}}\,.
\end{equation} 

In the proof we consider the binomial random graph $G(n,p)$ and the theorem is a consequence of the following two claims, which we verify below.

\begin{claim}\label{claim_cycles_2}
$\pr\big( \girth(G(n,p)) \geq k \big) \geq \exp( - kD_p^{k-1} n)$.
\end{claim}

\begin{claim}\label{claim_cycles_1}
$\pr\big( G(n,p) \rightarrow (C_k)_r \big) \geq 1- \exp( -  \frac{p}{2R^2} \binom{n}{2})$.
\end{claim}
In fact, since our choice of constants guarantees $p(n-1)>4R^2kD_p^{k-1}$, it follows from Claims~\ref{claim_cycles_2} and~\ref{claim_cycles_1} that 
with positive probability the random graph $G(n,p)$ has girth at least $k$ and the Ramsey-property for $C_k$ and $r$ colours.
Consequently, Theorem~\ref{thm_cycles} follows from our choice of~$n$ in~\eqref{eq:cyclesn}.
\end{proof}

\begin{proof}[Proof of Claim~\ref{claim_cycles_2}]
For the lower bound of the probability that~$G(n,p)$ has girth at least~$k$, we will use the FKG-inequality 
(see, e.g.,~\cite{JLR}*{Section~2.2}). For this purpose, let $X_{k-1}$ be the random variable of the number
of cycles of length less than~$k$. Clearly, $X_{k-1}$ is the sum of monotone increasing 
indicator random variables and,  hence, the FKG-inequality asserts
\begin{equation}\label{eq:cycFKG}
	\pr\big( \girth(G(n,p)) \geq k \big) 
	=
	\pr\big( X_{k-1}=0 \big)
	\overset{\text{\,FKG\,}}{\geq}
	\prod_{j=3}^{k-1}(1-p^j)^{\frac{(j-1)!}{2}\binom{n}{j}}
	\geq
	\exp\left(-\frac{\ex[X_{k-1}]}{1-p^3}\right),
\end{equation}
where we used the estimate $1-x\geq \exp(-x/(1-x))$ for the last inequality.
Since $pn\geq 1$ we have 
\[
	\ex[X_{k-1}]
	=
	\sum_{j=3}^{k-1}\frac{(j-1)!}{2}\binom{n}{j}p^j
	\leq
	\sum_{j=3}^{k-1} \frac{(pn)^j}{2j}
 	\leq 
	\frac{k}{6}(pn)^{k-1}  
	\overset{\eqref{eq:cycpt}}{=} 
	\frac{k}{6}D_p^{k-1} n
\]
and the claim follows from $1-p^3>1/6$ and~\eqref{eq:cycFKG}.
\end{proof}

\begin{proof}[Proof of Claim~\ref{claim_cycles_1}]
We will apply the Container Lemma to the $k$-uniform hypergraph $\cH=\binom{K_n}{C_k}$ which is the system 
of all cycles~$C_k$ of length~$k$ in~$K_n$, i.e., $V(\cH)=E(K_n)$ and~$k$ edges of $K_n$ correspond to a hyperedge in $\cH$, if
they form a cycle of length~$k$. For the application of Theorem~\ref{thm:cl} we first verify~\eqref{containhyp}. In that direction
we note 
\[
	d_1
	=
	\frac{k\cdot |E(\cH)|}{|V(\cH)|}
	=
	\frac{k!\binom{n}{k}}{2\binom{n}{2}}
	\geq
	\frac{k!}{k^k}n^{k-2}\,,
	\qquad
	d_j\leq n^{k-j-1}
\]
for $j=2,\dots,k-1$, and $d_k=1$. Therefore, for $j=2,\dots,k-1$  we have
\[
	\frac{d_j}{d_1\cdot \tau^{j-1}}
	\leq
	\frac{k^k}{k!\cdot (\tau n)^{j-1}}
	\leq 
	\frac{k^k}{k!\cdot\tau n}
	=
	\frac{k^k}{k!\cdot D_\tau n^{1/(k-1)}}
	\leq
	\frac{k^k}{k!\cdot n^{1/(k-1)}}
\]
and, moreover, 
\[
	\frac{d_k}{d_1\cdot \tau^{k-1}}
	\leq 
	\frac{k^k}{k!\,n^{k-2}\cdot \tau^{k-1}}
	=
	\frac{k^k}{k!\cdot D_\tau^{k-1}}\,.
\]
Combining both estimates yields
\[
	\frac{6\cdot k!\cdot 2^{\binom{k}{2}}}{d_1}\sum_{j=2}^k\frac{d_j}{2^{\binom{j-1}{2}}\tau^{j-1}}
	\leq
	\frac{6\cdot k^{k+1}\cdot 2^{\binom{k}{2}}}{\min\{n^{1/(k-1)}\,,D_\tau^{k-1}\}}
	\overset{\eqref{eq:cyclesepsDtau},\eqref{eq:cyclesn}}{\leq} 
	\eps.
\]
Having verified~\eqref{containhyp} of Theorem~\ref{thm:cl}, we infer properties~\ref{cl:i}--\ref{cl:iii} for every independent 
set $I\subseteq V(\cH)$. We consider the family~$\cB$ of all graphs~$B\subseteq K_n$ that fail to have the Ramsey property, i.e., 
$B \not \rightarrow (C_k)_r$. Below we establish Claim~\ref{claim_cycles_1} by showing
\[
	\pr\big(G(n,p) \in \mathcal{B}\big)\leq\exp\left( -  \frac{p}{2R^2} \binom{n}{2}\right).
\]
By the definition of~$\mathcal{B}$, for every $B\in\cB$ there exists a 
partition~$E(B)=I^B_1\dcup\dots\dcup I^B_r$ with the property that none of the sets~$I^B_i$ contains a cycle~$C_k$. In particular, 
each~$I^B_i$ is an independent set in~$\cH$ and, therefore, properties~\ref{cl:i}--\ref{cl:iii} of the Container Lemma assert that for every $i\in[r]$ there exists an $s$-tuple $\bcS^B_i=(\cS_{i,1}^B,\dots,\cS_{i,s}^B)$
of subsets of $I_i^B$ and a container set $\cC(\bcS^B_i)\supseteq I_i^B$ such that 
\begin{equation*}
	\big|\cS_{i,\sigma}^B\big|\leq \tau K\binom{n}{2}
\end{equation*}
for every $\sigma\in[s]$ and
\begin{equation}\label{eq:clcycp2}
	\big|e(\cC(\bcS^B_i))\big|\leq \eps \cdot e(\cH) =\eps\cdot\frac{(k-1)!}{2}\binom{n}{k}\,.
\end{equation}
We also set $\ccS^B=(\bcS^B_1,\dots,\bcS^B_r)$ and $\ccC^B=(\cC(\bcS^B_1),\dots,\cC(\bcS^B_r))$.

Moreover, for any possible $r$-tuple $\ccS=(\bcS_1,\dots,\bcS_r)$
of $s$-tuples of sets of size at most~$\tau K\binom{n}{2}$
we consider the corresponding container vector $\ccC(\ccS)=(\cC(\bcS_1),\dots,\cC(\bcS_r))$
given by the Container Lemma. We denote by~$D(\ccS)$ its complement in~$E(K_n)$ given by
\[
	D(\ccS)=E(K_n)\setminus\big(\cC(\bcS_1)\cup\dots\cup\cC(\bcS_r)\big)\,.
\]
We observe that for any $B\in\cB$ the following two properties hold:
\begin{enumerate}[label=\alabel]
\item\label{it:clcyca} $\bigcup_{i\in[r]}\bigcup_{\sigma\in[s]}\cS_{i,\sigma}^B\subseteq E(B)$ and
\item\label{it:clcycb}  $E(B) \cap D(\ccS^B)=\emptyset$,
since 
\[
	E(B) = I^B_1\dcup\dots\dcup I^B_r \subseteq \cC(\bcS^B_1)\cup\dots\cup\cC(\bcS^B_r)=E(K_n)\setminus D(\ccS^B)\,.
\] 
\end{enumerate}
From~\ref{it:clcyca} and~\ref{it:clcycb} we infer that
\begin{align}
	\pr\big(G(n,p)\in\cB\big) &\leq \sum_{\ccS=(\bcS_1,\dots,\bcS_r)}p^{\left|\bigcup_{i\in[r]}\bigcup_{\sigma\in[s]}\cS_{i,\sigma}\right|}
		\cdot\pr\Big( E\big(G(n,p)\big)\cap D(\ccS)=\emptyset\Big) \nonumber\\
		&\leq \max_{\ccS}\pr\Big(E\big(G(n,p)\big)\cap D(\ccS)=\emptyset\Big)
		\sum_{\ccS}p^{\left|\bigcup_{i\in[r]}\bigcup_{\sigma\in[s]}\cS_{i,\sigma}\right|}\label{eq:clgoalcyc}
		\,,
\end{align}
where the sum and the maximum are taken over all $r$-tuples $\ccS=(\bcS_1,\dots,\bcS_r)$
of $s$-tuples $\bcS_i=(\cS_{i,1},\dots,\cS_{i,s})$ 
of sets $\cS_{i,\sigma}$ of size at most~$\tau K\binom{n}{2}$ for $i\in[r]$ and $\sigma\in[s]$.
We will use property~\ref{cl:ii} of the Container Lemma to bound the maximum probability, 
while our choice of constants allow us to derive a sufficient bound for the sum.

For the maximum probability below we first observe that for every $\ccS=(\bcS_1,\dots,\bcS_r)$ we have 
\begin{equation}\label{eq:Ddensecyc}
	|D(\ccS)|>\frac{1}{R^2}\binom{n}{2}\,.
\end{equation}
For the proof we use the fact that for any $(r+1)$-colouring of $E(K_n)$ either there are more than
\[
	\frac{1}{2\binom{R}{k}}\binom{n}{k}
\] 
monochromatic copies of $C_k$ in the first $r$ colours or there are more than $\frac{1}{R^2}\binom{n}{2}$
edges having the last colour (see, e.g.,~\cite{RRS}*{Proposition~8} for the same assertion for cliques $K_k$ 
instead of cycles).

In view of this fact,  we consider $\cC(\bcS_1)\cup\dots\cup\cC(\bcS_r)\cup D(\ccS)$
as an $(r+1)$-colouring of~$E(K_n)$.
Owing to property~\ref{cl:ii} of the Container Lemma (see~\eqref{eq:clcycp2}),
every $\cC(\bcS_i)$ contains at most~$\eps\frac{(k-1)!}{2}\binom{n}{k}$
copies of $C_k$ and, hence, there are at most 
\[
	r\cdot\eps\frac{(k-1)!}{2}\binom{n}{k}\overset{\eqref{eq:cyclesepsDtau}}{\leq}\frac{1}{2\binom{R}{k}}\binom{n}{k}
\]
monochromatic copies in the first $r$ colours. Therefore, the mentioned fact above yields~\eqref{eq:Ddensecyc}
and, consequently, we arrive at  
\begin{equation}\label{eq:clmaxPcyc}
	\pr\Big( E\big(G(n,p)\big)\cap D(\ccS)=\emptyset \Big) 
	= 
	(1-p)^{|D(\ccS)|}
	\overset{\eqref{eq:Ddensecyc}}{\leq}
	\exp\bigg(-\frac{p}{R^2}\binom{n}{2}\bigg)
\end{equation}
for every $\ccS$ considered here.  In particular,~\eqref{eq:clmaxPcyc} bounds the maximum probability 
considered in the R-H-S of~\eqref{eq:clgoalcyc} and below we turn to the sum in~\eqref{eq:clgoalcyc}. 

Owing to $|\cS_{i,\sigma}|\leq \tau K\binom{n}{2}$ for every $i\in[r]$ and $\sigma\in[s]$
we have 
\[
	 \sum_{\ccS=(\bcS_1,\dots,\bcS_r)}p^{\left|\bigcup_{i\in[r]}\bigcup_{\sigma\in[s]}\cS_{i,\sigma}\right|}
	 \leq
	 \sum_{m=0}^{r\cdot  s\cdot \tau K\binom{n}{2}}\binom{\binom{n}{2}}{m}2^{r s m}p^m
	 \leq
	 \sum_{m=0}^{r\cdot  s\cdot \tau K\binom{n}{2}}\left(\frac{\eu\binom{n}{2}}{m}2^{rs} p\right)^m\,.
\]
Since the function $m\mapsto (\eu \binom{n}{2}2^{rs}p/m)^m$ is unimodal and attains its maximum value
for $m_0=\binom{n}{2}2^{rs}p\geq r s\tau K\binom{n}{2}$, we can bound the summands in R-H-S above 
by the last one and obtain
\begin{align}
	\sum_{\ccS=(\bcS_1,\dots,\bcS_r)}p^{\left|\bigcup_{i\in[r]}\bigcup_{\sigma\in[s]}\cS_{i,\sigma}\right|}
	&\overset{\phantom{\eqref{eq:clDpcyc}}}{\leq}
	\left(r  s \tau K\binom{n}{2}+1\right)\cdot \left(\frac{\eu\binom{n}{2}}{r  s \tau K\binom{n}{2}}2^{rs} p\right)^{r  s \tau K\binom{n}{2}}\nonumber\\
	&\overset{\phantom{\eqref{eq:clDpcyc}}}{\leq}
	n^2\cdot\left(\frac{2^{rs} \eu D_p}{r sKD_{\tau}}\right)^{r  s \tau K \binom{n}{2}}\nonumber\\
	&\overset{\eqref{eq:clDpcyc}}{=}
	n^2\cdot\big(2^{rs}\eu\cdot 10R^2rs\log(10R^2r)\big)^{r  s \tau K\binom{n}{2}}\nonumber\\
	&\overset{\phantom{\eqref{eq:clDpcyc}}}{\leq}
	n^2\cdot\exp\left(r  s \tau K\binom{n}{2}\Big(rs+1+\ln\big(10R^2rs\log(10R^2r)\big)\Big)\right)\nonumber\\
	&\overset{{\eqref{eq:clDpcyc}}}{\leq}
	n^2\cdot\exp\bigg(\frac{p}{3R^2}\binom{n}{2}\bigg)\nonumber\\
	&\overset{\phantom{\eqref{eq:clDpcyc}}}{\leq}
	\exp\bigg(\frac{p}{2R^2}\binom{n}{2}\bigg)\,.\label{eq:clccScyc}
\end{align}
Finally, combining~\eqref{eq:clmaxPcyc} and~\eqref{eq:clccScyc} 
with~\eqref{eq:clgoalcyc} leads to 
\[
	\pr\big(G(n,p)\in\cB\big)
	\leq 
	\exp\bigg(-\frac{p}{R^2}\binom{n}{2}\bigg) \cdot \exp\bigg(\frac{p}{2R^2}\binom{n}{2}\bigg)
	=
	\exp\bigg(-\frac{p}{2R^2}\binom{n}{2}\bigg)\,,
\]
which concludes the proof of the claim.
\end{proof}

\section{Proof of Theorem~\ref{thm_ap}}\label{sec_ap}

We prove Theorem~\ref{thm_ap} by establishing that, for given integers~$k\geq 3$, $r\geq 2$, 
and~$g \geq 2$, there exists a set~$S \subseteq \NN$ of size at most~$k^{40k^2(k+g)}W^{12k(k+g)}$ 
(where $W=\vdW(k;r)$ is the van der Waerden number guaranteeing 
monochromatic arithmetic progressions of length~$k$ for any $r$-colouring of $[W]$)
such that the hypergraph~$\binom{S}{\kap}$ has chromatic number greater than~$r$, and girth at least~$g$.

\begin{proof}
Let~$k\geq 3$, $r\geq 2$, and~$g \geq 2$ be given and let~$W=\vdW(k;r)$ be the van der Waerden number. 
We first define all constants involved in the proof. 
For the application of the Container Lemma we 
set 
\begin{equation}
	\label{eq:apepsDtau}
	\eps = \frac{1}{rW^3}
	\qqand
	D_\tau=\left(\frac{6\cdot k!\cdot 2^{\binom{k}{2}}\cdot k^3}{\eps}\right)^{\frac{1}{k-1}}
\end{equation}
and we fix integers
\begin{equation}\label{eq:apKs}
	K=800k\big(k!\big)^3<30k^{3k}
	\qqand 
	s=\lfloor K\log(1/\eps)\rfloor
	<30k^{3k}\log(rW^3)\,.
\end{equation}
We set
\begin{equation}\label{eq:apDp}
	D_p=128Wr^2 s^2KD_\tau\log(128Wr)<2^{40}k^{10k}r^3W^{3}
\end{equation}
and define 
\begin{equation}
	\label{eq:apn}
	n
	=
	k^{4g}D_p^{2k(k+g)}
	<
	k^{40k^2(k+g)}W^{12k(k+g)}\,.
\end{equation}
Finally, we define the following parameters appearing in the proof
\begin{equation}
	\label{eq:appt}
	\tau=D_\tau n^{-\frac{1}{k-1}}\,,\qquad
	p=D_pn^{-\frac{1}{k-1}}\,,
	\qqand
	t=\frac{pn}{8W}
	\,.
\end{equation}

Let~$[n]_p$ denote the random set obtained by choosing each element of~$[n]=\{1,2,\dots,n\}$ independently with probability~$p$. The theorem is an immediate consequence of the following two claims.

\begin{claim}\label{claim_ap_1}
With probability larger than~$1/2$ for the random subset $[n]_p$ there
exists a set~$T\subseteq [n]_p$
of size at most $t$ such that such that~$\girth\binom{[n]_p\setminus T}{\kap}\geq g$.
\end{claim}

\begin{claim}\label{claim_ap_2}
With probability larger than~$1/2$ the random subset $[n]_p$ satisfies~$\chi\binom{[n]_p\setminus T}{\kap}>r$ for every 
subset~$T \subseteq [n]_p$ of size at most $t$.
\end{claim}

Together, these claims establish that, with positive probability, the random set~$[n]_p$ will have the property that there exists a set~$T\subseteq [n]_p$ of size~$t$ so that the hypergraph ${[n]_p \setminus T \choose \kap}$ has girth at least~$g$ and 
chromatic number bigger than~$r$. Thus, these claims together 
with our choice of $n$ in~\eqref{eq:apn} establish the existence of a set $S\subseteq\NN$ 
as claimed in~Theorem~\ref{thm_ap}. We remark that, such a set will likely have only $O(pn)$ 
elements (not~$n$ elements). However, this improvement is negligible. 
\end{proof}

\begin{proof}[Proof of Claim~\ref{claim_ap_1}]
The proof follows by a standard first moment argument.
Recall that a $2$-cycle in a hypergraph consists of two hyperedges sharing at least two vertices and 
for $j>2$ a $j$-cycle consists of a cyclically ordered sequence of hyperedges~$e_1,e_2,\dots,e_j$ where the 
intersection of two consecutive edges is exactly~$1$, the intersection of any two nonconsecutive edges 
is empty, and the intersection points for each pair of consecutive edges is unique (which for $j\geq 4$ is already implied by the other two conditions).  
Let the random variable~$X_j$ denote the number of~$j$-cycles appearing in the random 
hypergraph~$\binom{[n]_p}{\kap}$.

We first estimate $\ex[X_2]$. Since the hyperedges of $\binom{[n]}{\kap}$ are arithmetic progressions
of length~$k$,
every pair of vertices is contained in at most $\binom{k}{2}$ such hyperedges. 
Consequently, we have
\[
	\ex[X_2] 
	\leq 
	\binom{n}{2}\binom{k}{2}^2p^{k+1} 
	\leq k^4 p^{k+1}n^2 
	= 
	pn\cdot \frac{k^4D_p^{k}}{n^{1/(k-1)}}
	<
	\frac{t}{4}\,.
\]

Next we bound $\ex[X_j]$ for $3\leq j< g$. For that we note that for any~$j$-cycle we may first 
select and order the~$j$ vertices of degree~2 and then fixing the remaining vertices of 
each edge. However, since every edge of the cycle contains two (already fixed) vertices of 
degree two, again there are at most $\binom{k}{2}$ possible completion for such an edge 
and, hence, we have 
\[
	\sum_{j=3}^{g-1}\ex[X_j]
	\leq
	\sum_{j=3}^{g-1}n^jk^{2j}p^{(k-1)j}
	=
	\sum_{j=3}^{g-1}k^{2j}D_p^{(k-1)j}
	<
	k^{2g}D_p^{kg}
	\leq 
	\frac{t}{4}\,.
\]
By Markov's inequality this implies that with probability less than~$1/2$
the randomly generated hypergraph~$\binom{[n]_p}{\kap}$ 
has at most~$t$ or more cycles of length less than~$g$, which establishes
Claim~\ref{claim_ap_1}.
\end{proof}

\begin{proof}[Proof of Claim~\ref{claim_ap_2}]
We consider the $k$-uniform hypergraph~$\mathcal{H}=\binom{[n]}{\kap}$ 
and check that it satisfies the assumptions of the Container Lemma (Theorem~\ref{thm:cl}) 
for the parameters~$\eps$ and~$\tau$ chosen in~\eqref{eq:apepsDtau} and~\eqref{eq:appt}.
Note that~$\epsilon < 1/2$ by definition and $\tau<1/2$ follows from the 
choice of $n$ in~\eqref{eq:apn}. For the remaining assumption~\eqref{containhyp} 
we recall the definition of the average degrees~$d_j$ for $j=1,\dots,k$ of $\cH$
and again, using the fact that every pair of vertices is contained in at most $\binom{k}{2}$
$\kap$'s, we note that for $j=2,\dots,k$
we have
\[
	d_j\leq d_2\leq\binom{k}{2}<\frac{k^2}{2}\,.
\]
Moreover, we have
\begin{equation}\label{eq:nAPk}
	d_1
	=
	\frac{k}{n}\cdot\left|\binom{[n]}{\kap}\right|
	=
	\frac{k}{n}\cdot \sum_{i=1}^{n-k+1} \left\lfloor \frac{n-i}{k-1}\right\rfloor
	=
	\frac{k}{n}\cdot \sum_{i=1}^{n-1} \left\lfloor \frac{n-i}{k-1}\right\rfloor
	\geq
	\frac{k}{n}\cdot \sum_{i=1}^{n-1} \left(\frac{n-i}{k-1}-1\right)
	\geq
	\frac{n}{2}\,.
\end{equation}
Consequently,
\[
	\frac{6 \cdot k! \cdot 2^{\binom{k}{2}}}{d_1} 
		\sum_{j=2}^k \frac{d_j}{2^{\binom{j-1}{2}}\tau^{j-1}} 
	\leq
	\sum_{j=2}^{k}\frac{6\cdot k!\cdot 2^{\binom{k}{2}}\cdot k^2}{n\cdot 2^{\binom{j-1}{2}}\cdot D_{\tau}^{j-1}n^{-\frac{j-1}{k-1}}}
	<
	\frac{6\cdot k!\cdot 2^{\binom{k}{2}}\cdot k^3}{\min\{D_{\tau}^{k-1},\,n^{\frac{1}{k-1}}\}}
	\overset{\eqref{eq:apepsDtau},\eqref{eq:apn}}{\leq}
	\eps.
\]
This shows that condition~\eqref{containhyp} of Theorem~\ref{thm:cl}  holds. 
Consequently, for every independent set $I\subseteq V(\cH)$ we can apply 
conclusions~\ref{cl:i}--\ref{cl:iii} of the Container Lemma 
with the constants defined in~\eqref{eq:apKs}.

We consider the family~$\cB$ of all sets~$B\subseteq [n]$ 
with the property that there exists a set~$T \subseteq E(B)$ of size at most~$t$ such that~$(B\setminus T) \not \rightarrow (\kap)_r$, i.e., 
$\chi\binom{B\setminus T}{\kap}\leq r$. Claim~\ref{claim_ap_2} is equivalent to
\[
	\pr\big([n]_p \in \mathcal{B}\big)<\frac{1}{2}\,.
\]
By definition of~$\mathcal{B}$, for every $B\in\cB$ there exists a set~$T^B\subseteq B$ of size~$|T^B| \leq t$ 
and a partition~$B=I^B_1\dcup\dots\dcup I^B_r\dcup T^B$ with the property that none of the sets~$I^B_i$ contains 
an~$\kap$. In particular, 
each~$I^B_i$ is an independent set in~$\cH$ and, therefore, properties~\ref{cl:i}--\ref{cl:iii} of the Container Lemma assert
that for every $i\in[r]$ there exists an $s$-tuple $\bcS^B_i=(\cS_{i,1}^B,\dots,\cS_{i,s}^B)$
of subsets of $I_i^B$ and a container set $\cC(\bcS^B_i)\supseteq I_i^B$ such that 
\begin{equation}
	\big|\cS_{i,\sigma}^B\big|\leq \tau Kn
\end{equation}
for every $\sigma\in[s]$ and
\begin{equation}\label{eq:app2}
	\big|e_{\cH}(\cC(\bcS^B_i))\big|\leq \eps\cdot e(\cH)=\eps\left|\binom{[n]}{\kap}\right|\,.
\end{equation}
We set $\ccS^B=(\bcS^B_1,\dots,\bcS^B_r)$ and $\ccC^B=(\cC(\bcS^B_1),\dots,\cC(\bcS^B_r))$.

Moreover, for any possible $r$-tuple $\ccS=(\bcS_1,\dots,\bcS_r)$
of $s$-tuples of sets of size at most~$\tau Kn$,
we consider the corresponding container vector $\ccC(\ccS)=(\cC(\bcS_1),\dots,\cC(\bcS_r))$
given by the Container Lemma. We denote by~$D(\ccS)$ its complement in~$[n]$ given by
\[
	D(\ccS)=[n]\setminus\big(\cC(\bcS_1)\cup\dots\cup\cC(\bcS_r)\big)\,.
\]
Observe that for any $B\in\cB$ the following two properties hold:
\begin{enumerate}[label=\alabel]
\item\label{it:apa} $\bigcup_{i\in[r]}\bigcup_{\sigma\in[s]}\cS_{i,\sigma}^B\subseteq\big(E(B)\setminus D(\ccS^B)\big)$ and
\item\label{it:apb}  $\big|B \cap D(\ccS^B)\big| \leq |T^B|\leq t$,
since 
\[
	B \setminus T^B = I^B_1\dcup\dots\dcup I^B_r \subseteq \cC(\bcS^B_1)\cup\dots\cup\cC(\bcS^B_r)=[n]\setminus D(\ccS^B)\,.
\] 
\end{enumerate}
From~\ref{it:apa} and~\ref{it:apb} we infer that
\begin{align}
	\pr\big([n]_p\in\cB\big) &\leq \sum_{\ccS=(\bcS_1,\dots,\bcS_r)}p^{\left|\bigcup_{i\in[r]}\bigcup_{\sigma\in[s]}\cS_{i,\sigma}\right|}
		\cdot\pr\big( \big|[n]_p\cap D(\ccS)\big|\leq t\big) \nonumber\\
		&\leq \max_{\ccS}\pr\big( \big|[n]_p\cap D(\ccS)\big|\leq t\big)\cdot
		\sum_{\ccS}p^{\left|\bigcup_{i\in[r]}\bigcup_{\sigma\in[s]}\cS_{i,\sigma}\right|}\label{eq:apgoal}
		\,,
\end{align}
where the sum and the maximum are taken over all $r$-tuples $\ccS=(\bcS_1,\dots,\bcS_r)$
of $s$-tuples $\bcS_i=(\cS_{i,1},\dots,\cS_{i,s})$ of sets $\cS_{i,\sigma}$ of size at most~$\tau Kn$ for $i\in[r]$ and $\sigma\in[s]$.
We will use property~\ref{cl:ii} of the Container Lemma to bound the maximum probability in~\eqref{eq:apgoal}, 
while our choice of constants allow us to derive a sufficient bound for the sum in~\eqref{eq:apgoal}.
We shall use the following fact (the proof of which we defer to the end of this section).
\begin{fact}\label{fact:vdW}
	For every $(r+1)$-colouring of $[n]$ either there are more than 
	$\big|\binom{[n]}{\kap}\big|/W^3$ monochromatic $\kap$'s in 
	the first $r$ colours or more than $\frac{n}{4W}$ elements are in the last colour.
\end{fact} 
In view of this fact,  we consider $\cC(\bcS_1)\cup\dots\cup\cC(\bcS_r)\cup D(\ccS)$
as an $(r+1)$-colouring of~$[n]$.
Owing to property~\ref{cl:ii} of the Container Lemma (see~\eqref{eq:app2}),
every $\cC(\bcS_i)$ contains at most~$\eps\big|\binom{[n]}{\kap}\big|$
monochromatic $\kap$'s and, hence, there are at most 
\[
	r\cdot\eps\left|\binom{[n]}{\kap}\right|\overset{\eqref{eq:apepsDtau}}{=}\frac{1}{W^3}\left|\binom{[n]}{\kap}\right|
\]
monochromatic $\kap$'s in the first $r$ colours. Therefore, Fact~\ref{fact:vdW}  yields
that for every $\ccS=(\bcS_1,\dots,\bcS_r)$ we have 
\begin{equation}\label{eq:Ddenseap}
	|D(\ccS)|>\frac{n}{4W}\,.
\end{equation}
In particular, the choice of $t$ combined with~\eqref{eq:Ddenseap}
yields $t < p|D(\ccS)|/2$ and, consequently, Chernoff's inequality (see, e.g.,~\cite{JLR}*{Theorem~2.1}
asserts 
\begin{equation}\label{eq:apmaxP}
	\pr\Big( \big|[n]_p\cap D(\ccS)\big|\leq t \Big) \leq \exp\bigg(-\frac{pn}{32W}\bigg)
\end{equation}
for every $\ccS$ considered here.  In particular,~\eqref{eq:apmaxP} bounds the maximum probability 
considered in the R-H-S of~\eqref{eq:apgoal} and below we turn to the sum in~\eqref{eq:apgoal}. 

Owing to $|\cS_{i,\sigma}|\leq \tau Kn$ for every $i\in[r]$ and $\sigma\in[s]$
we have 
\begin{equation}\label{eq:apkx}
	 \sum_{\ccS=(\bcS_1,\dots,\bcS_r)}p^{\left|\bigcup_{i\in[r]}\bigcup_{\sigma\in[s]}\cS_{i,\sigma}\right|}
	 \leq
	 \sum_{m=0}^{r\cdot  s\cdot \tau Kn}\binom{n}{m}2^{rs m}p^m
	 \leq
	 \sum_{m=0}^{r\cdot  s\cdot \tau Kn}\left(\frac{\eu n}{m}2^{rs} p\right)^m\,.
\end{equation}
Since the function $m\mapsto (\eu n2^{rs}p/m)^m$ is unimodal and attains its maximum value
for $m_0=2^{rs}pn\geq r s\tau Kn$, from~\eqref{eq:apkx} we obtain
\begin{align}
	\sum_{\ccS=(\bcS_1,\dots,\bcS_r)}p^{\left|\bigcup_{i\in[r]}\bigcup_{\sigma\in[s]}\cS_{i,\sigma}\right|}
	&\overset{\phantom{\eqref{eq:apDp}}}{\leq}
	\left(r  s \tau Kn+1\right)\cdot \left(\frac{\eu n}{r  s \tau Kn}2^{rs} p\right)^{r  s \tau Kn}\nonumber\\
	&\overset{\phantom{\eqref{eq:apDp}}}{\leq}
	n\cdot\left(\frac{2^{rs} \eu D_p}{r sKD_{\tau}}\right)^{r  s \tau K n}\nonumber\\
	&\overset{\eqref{eq:apDp}}{=}
	n\cdot\big(2^{rs+7}\eu\cdot Wrs\log(128Wr)\big)^{r  s \tau Kn}\nonumber\\
	&\overset{\phantom{\eqref{eq:apDp}}}{\leq}
	n\cdot\exp\left(r  s \tau Kn\Big(rs+6+\ln\big(Wrs\log(128Wr)\big)\Big)\right)\nonumber\\
	&\overset{{\eqref{eq:apDp}}}{\leq}
	n\cdot\exp\Big(\frac{pn}{128W}\Big)\nonumber\\
	&\overset{\phantom{\eqref{eq:apDp}}}{\leq}
	\exp\Big(\frac{pn}{64W}\Big)\,.\label{eq:apccS}
\end{align}
Finally, combining~\eqref{eq:apmaxP} and~\eqref{eq:apccS} 
with~\eqref{eq:apgoal} leads to 
\[
	\pr\big([n]_p\in\cB\big)
	\leq 
	\exp\Big(-\frac{pn}{32W}\Big) \cdot \exp\Big(\frac{pn}{64W}\Big)
	=
	\exp\Big(-\frac{pn}{64W}\Big)
	<
	\frac{1}{2}\,.
\]
Up to the proof of Fact~\ref{fact:vdW} this concludes the proof of Claim~\ref{claim_ap_2}.
\end{proof}

\begin{proof}[Proof of Fact~\ref{fact:vdW}]
	Recall that $W=\vdW(k;r)$ asserts that $A\rightarrow(\kap)_r$ for every arithmetic progression $A\subseteq \NN$
	of length~$W$. Consider an arbitrary $(r+1)$-colouring of $[n]$. Suppose at most $\frac{n}{4W}$ elements of $[n]$
	receive colour $r+1$. From the observation that for every $w\leq W/2$ and 
	every $i\in [n]$ there are at most $\frac{n-1}{W-w}$ distinct $\Wap$'s in $[n]$
	having $i$ at position $w$ or $W-w+1$, one can deduce that every $i\in[n]$ is contained in at most~$n$ different 
	$\Wap$'s.
	Consequently, there are at least 
	\[
		\left|\binom{[n]}{\Wap}\right|-\frac{n^2}{4W}
		\geq 
		\frac{n^2}{4W}
	\]
	$\Wap$'s containing no element of the last colour, 
	where we used $\big|\binom{[n]}{\Wap}\big|\geq \frac{n^2}{2W}$ for the last inequality (cf.~\eqref{eq:nAPk}).
	
	Owing to the choice of $W$ every such $r$-coloured $\Wap$ contains a monochromatic $\kap$ in one of the 
	first~$r$ colours. On the other hand, every $\kap$ can be contained in at most~$\binom{W}{2}$
	different $\Wap$'s in $[n]$. Therefore, there exist at least
	\[
		\frac{n^2}{4W}\cdot \frac{2}{W^2} 
		= 
		\frac{n^2}{2W^3}
		\geq 
		\frac{1}{W^3}\cdot\left|\binom{[n]}{\kap}\right|
	\]
	distinct monochromatic $\kap$'s in $[n]$ coloured in one of the first $r$ colours and the fact follows.
\end{proof}

\section{Proof of Theorem~\ref{thm_cliques}} \label{sec_cliques}
In this section we establish that for all integers~$r \geq 2$,~$g \geq 3$, and~$k \geq 3$, there exists a graph~$H$ on at most~$k^{40gk^4}R^{40gk^2}$ 
vertices such that~${{H} \choose {K_k}}$ has chromatic number greater than~$r$ and girth at least~$g$, where~$R=R(K_k;r)$
is the $r$-colour Ramsey number for $K_k$.

\begin{proof}
We first define all constants involved in the proof. Given the uniformity $k\geq 3$, the number of colours $r\geq 2$, and the minimum girth $g\geq 3$, we denote 
by~$R=R(K_k;r)$ the $r$-colour Ramsey number for $K_k$. In the estimates below we sometimes use the 
trivial observation $2^r<R$. For a later application of the Container Lemma (Theorem~\ref{thm:cl}), we define the involved auxiliary 
constants (and observe some immediate bounds)
\begin{equation}\label{eq:clconst}
	\epsilon = \frac{1}{2 r {R \choose k}}<\frac{1}{2}
	\qqand
	D_\tau
	=\left(\frac{6\cdot\binom{k}{2}!\cdot 2^{\binom{\binom{k}{2}}{2}}\binom{k}{2}k^k}{\eps}\right)^{10/k^2}
	<
	2^{3k^2/2}k^{20} R^{20/k}
		\end{equation}
and integers
\begin{equation}\label{eq:clcnt}
	K=800\tbinom{k}{2}\big(\tbinom{k}{2}!\big)^3<k^{3k^2}
	\qqand 
	s=\lfloor K\log(1/\eps)\rfloor
	<k^{3k^2}\log(rR^k)\,.
\end{equation}
We introduce another auxiliary constant 
\begin{equation}\label{eq:clCp}
	D_p=50R^2r^2 s^2KD_\tau\log(50R^2r)
	<
	k^{10k^2+30}R^{5+20/k}
\end{equation}
and set
\begin{equation}\label{eq:cliquesn}
		n=D_p^{k^2(5+g)}
	<
	D_p^{3k^2g}
	<
	k^{40gk^4}R^{40gk^2}\,.
\end{equation}
Finally, we define the following three parameters in terms of some of the constants above
\begin{equation}\label{eq:cliquestpt}
	\tau=\frac{D_\tau}{n^{2/(k+1)}}\,,\qquad
	p=\frac{D_p}{n^{2/(k+1)}}\,,\qqand
	t=\frac{p}{2R^2}\binom{n}{2}\,.
\end{equation}
Having defined all involved constants we shall show the following two claims, which yield the theorem.
\begin{claim}\label{claim_cliques_1}
With probability larger than~$1/2$, the random graph $G(n,p)$ has the property that there is a set $T\subseteq E(G(n,p))$ of size at most $t$ such that~$\girth{G(n,p)-T \choose {K_k}}\geq g$.
\end{claim}

\begin{claim}\label{claim_cliques_2}
With probability larger than~$1/2$, the random graph $G(n,p)$ 
has the property  that~$\chi{ { {G(n,p)-T } \choose {K_k} }  > r}$ for every subset~$T \subseteq E(G(n,p))$ of size at most $t$.
\end{claim}

Both claims together show that with positive probability there exists a graph~$G$ which contains a set $T\subseteq E(G)$ of size at most $t$
such that $H=G-T$ satisfies $\girth\binom{H}{K_k}\geq g$ and $\chi\binom{H}{K_k}>r$. Consequently, Theorem~\ref{thm_cliques}
follows from the choice of~$n$ in~\eqref{eq:cliquesn}.
\end{proof}

\begin{proof}[Proof of Claim~\ref{claim_cliques_1}]

Recall that a 2-cycle is a pair of edges~${e_1,e_2}$ such that~$|e_1 \cap e_2|>1$ and for~$j>2$ a~$j$-cycle is a cyclical sequence of~$j$ edges~$e_1,e_2,\dots,e_j$ where the intersection of two consecutive edges is exactly one i.e.~$|e_i \cap e_{i+1}|=1$ (addition mod~$j$), the intersection of any two nonconsecutive edges is empty, and the intersection points for each pair of consecutive edges is unique. 

Define~$X_j$ to be the number of~$j$-cycles in the system of copies of~$K_k$ in~$G(n,p)$. We first work to bound~$X_2$. If~$k=3$, we trivially have~$\ex[X_2]=0.$  Otherwise for~$k \geq 4$, a 2-cycle corresponds to two copies of~$K_k$ that intersect in more than two edges, and thus in more than two vertices. Furthermore, we see that two copies 
of~$K_k$ that intersect in~$i$ vertices together span exactly~$2k-i$ vertices and~${2{k \choose 2}-{i \choose 2}}$ edges. With this in mind, the following bounds~$\ex[X_2] < t/4$ in~${{G(n,p)} \choose {K_k} }$:

\begin{align*}
\frac{\ex[X_2]}{t/4} &= \frac{8 R^2}{p {n \choose 2}}  \cdot  \ex[X_2]  \leq \frac{32R^2}{pn^2} \cdot \sum_{i=3}^{k-1} n^{2k-i} p^{2{k \choose 2}-{i \choose 2}} \\
&= 32R^2n^{2k-2}p^{2{k \choose 2} -1} \sum_{i=3}^{k-1} n^{(i^2-2i-ki)/(k+1)}D_p^{- {i \choose 2}} \\
&\leq 32 R^2 n^{2k-2}p^{2{k \choose 2} -1} \cdot k \cdot \max_{3 \leq i \leq k-1} \left\{ n^{(i^2-2i-ki)/(k+1)}\right\} \\
&\leq 32 R^2 n^{2k-2}p^{2{k \choose 2} -1} \cdot k  \cdot n^{(3-3k)/(k+1)} \\
&= \frac{32 k R^2  D_p^{k^2-k-1}}{n^{(k-3)/(k+1)}} \leq \frac{D_p^{k^2}}{n^{1/5}}\overset{\eqref{eq:cliquesn}}{<}1\,.\\
\end{align*}

We now bound~$\sum_{j=3}^{g-1} X_j$. For~$j>2$, a~$j$-cycle in~${K_n \choose {K_k} }$ consists of a cyclically ordered set of~$j$ copies of~$K_k$ such that each two consecutive copies intersect in exactly one edge of~$K_n$. Thus, a~$j$-cycle corresponds to a set of $K_k$'s in~$K_n$ that span at most~$kj-2j$ vertices in~$K_n$ and exactly~${{k \choose 2}j-j}$ edges in~$K_n$. From this, we see that, for~$2 < j < g$, we have
\[\ex[X_j] \leq n^{kj-2j} p^{{k \choose 2}j-j} = \left( n^{k-2} p^{{k \choose 2}-1} \right)^j = D_p^{({{k} \choose 2} -1)j}.\]
Using this, we establish~$\sum_{j=3}^{g-1} \ex[X_j] < t/4$:

\[
\frac{\sum_{j=3}^{g-1} \ex[X_j]}{t/4} \leq \frac{8R^2}{p{n \choose 2}} \cdot g  \cdot D_p^{({{k} \choose 2} -1)g} 
\leq \frac{32R^2 g}{pn^2} D_p^{({{k} \choose 2} -1)g} \leq \frac{D_p^{k^2g}}{n} < 1.
\]

Thus, we have shown~$\sum_{j=2}^{g-1} \ex[X_j] < t/4 + t/4 = t/2$. By Markov's inequality, this gives that, with probability bigger~$1/2$, the hypergraph~$\binom{G(n,p)}{K_k}$ contains less than~$t$ cycles of length less than~$g$. For each such cycle, removing one vertex (which is an edge in $G(n,p)$)
concludes the proof of Claim~\ref{claim_cliques_1}.
\end{proof}

\begin{proof}[Proof of Claim~\ref{claim_cliques_2}]

The proof relies on an application of the Container Lemma (Theorem~\ref{thm:cl}) to the $\binom{k}{2}$-uniform hypergraph~$\mathcal{H}={{K_n} \choose {K_k} }$ 
(for similar proofs see, e.g.,~\cites{NS,RRS}). In view of that we will first verify condition~\eqref{containhyp} for our 
choices of $\eps$ and $\tau$ in~\eqref{eq:clconst} and~\eqref{eq:cliquestpt}.
Recalling the definition of the average degrees $d_j$ for $j=1,\dots,\binom{k}{2}$ of $\cH$, we note that 
\[
	d_1=\binom{n-2}{k-2}\geq \frac{n^{k-2}}{k^k}\,.
\]
For $j\geq 2$, letting $k_j$ be the smallest integer such that $j\leq\binom{k_j}{2}$,
we have
\[
	d_j\leq\binom{n-k_j}{k-k_j}\leq n^{k-k_j}\,.
\]
Consequently, for every $j=2,\dots,\binom{k}{2}$ this gives
\begin{equation}\label{eq:d1j}
	\frac{d_j}{d_1\cdot\tau^{j-1}}
	\leq 
	\frac{k^k\cdot n^{2-k_j}\cdot n^{\frac{2j-2}{k+1}}}{D_\tau^{j-1}}
	\leq
	\frac{k^k\cdot n^{2-k_j}\cdot n^{\frac{2\binom{k_j}{2}-2}{k+1}}}{D_\tau^{j-1}}
	=
	\frac{k^k\cdot n^{\frac{(k_j-2)(k_j-k)}{k+1}}}{D_\tau^{j-1}}.
\end{equation}
For $k_j=k$, i.e., for $j=\binom{k-1}{2}+1,\dots,\binom{k}{2}$ we, therefore, get
\begin{equation}\label{eq:kjk}
	\frac{d_j}{d_1\cdot\tau^{j-1}} 
	\leq
	\frac{k^k}{D_\tau^{j-1}}
	\leq 
	\frac{k^k}{D_\tau^{\binom{k-1}{2}}}
	\leq 
	\frac{k^k}{D_\tau^{k^2/10}}\,,
\end{equation}
where we used $D_\tau\geq 1$ and $k\geq 3$ for the last inequalities. For integers $3\leq k_j\leq k-1$ 
we note that $k\geq 4$ and $(k_j-2)(k_j-k)$ is maximized for $k_j=3$ (and $k_j=k-1$).
Hence, in this case we can bound the R-H-S in~\eqref{eq:d1j} to give
\[
	\frac{d_j}{d_1\cdot\tau^{j-1}}
	\leq 
	\frac{k^k}{n^{1/5}}\,.
\]
Summarizing,  since $k_j\geq 3$ for $j\geq 2$ we arrive for $h=\binom{k}{2}$ at
\[
	\frac{6 \cdot h! \cdot 2^{\binom{h}{2}}}{d_1} \sum_{j=2}^h \frac{d_j}{2^{j-1 \choose 2}\tau^{j-1}} 
	\leq
	\frac{6 \cdot h! \cdot 2^{\binom{h}{2}}\cdot h\cdot k^k}{\min\{D_\tau^{k^2/10}\,,n^{1/5}\}}
	\overset{\eqref{eq:clconst},\eqref{eq:cliquesn}}{\leq} 
	\eps
\]
and this shows that condition~\eqref{containhyp} of Theorem~\ref{thm:cl}  holds. 
Consequently, for every independent set $I\subseteq V(\cH)$ we can apply 
conclusions~\ref{cl:i}--\ref{cl:iii} of the Container Lemma 
with the constants defined above.

We consider the family~$\cB$ of all graphs~$B\subseteq K_n$ such that there exists a set~$T \subseteq E(B)$ 
of size at most~$t$ and~$(B-T) \not \rightarrow (K_k)_r$, i.e.,  there exists and $r$-colouring of the edges of the graph~$B-T$
without a monochromatic copy of~$K_k$. In other words, $\chi\binom{B- T}{K_k}\leq r$ and we may view~$\cB$ as the set of all (`bad') graphs on~$n$ vertices that do not have the desired property of Claim~\ref{claim_cliques_2}.
Below we establish Claim~\ref{claim_cliques_2} by showing
\[
	\pr\big(G(n,p) \in \mathcal{B}\big)<\frac{1}{2}\,.
\]

Consider any graph~$B \in \mathcal{B}$. By the definition of~$\mathcal{B}$, there exists a set~$T^B\subseteq E(B)$ of size~$|T^B| \leq t$ 
and a partition~$E(B)\setminus T^B=I^B_1\dcup\dots\dcup I^B_r$ with the property that none of the sets~$I^B_i$ contains a~$K_k$. In particular, 
each~$I^B_i$ is an independent set in~$\cH$ and, therefore, properties~\ref{cl:i}--\ref{cl:iii} of the Container Lemma assert
that for every $i\in[r]$ there exists an $s$-tuple $\bcS^B_i=(\cS_{i,1}^B,\dots,\cS_{i,s}^B)$
of subsets of $I_i^B$ and a container set $\cC(\bcS^B_i)\supseteq I_i^B$ such that 
\begin{equation}
	\big|\cS_{i,\sigma}^B\big|\leq \tau K\binom{n}{2}
\end{equation}
for every $\sigma\in[s]$ and
\begin{equation}\label{eq:clp2}
	\big|e_{\cH}(\cC(\bcS^B_i))\big|\leq \eps \binom{n}{k}\,.
\end{equation}
We also set $\ccS^B=(\bcS^B_1,\dots,\bcS^B_r)$ and $\ccC^B=(\cC(\bcS^B_1),\dots,\cC(\bcS^B_r))$.

Moreover, for any possible $r$-tuple $\ccS=(\bcS_1,\dots,\bcS_r)$
of $s$-tuples of sets of size at most~$\tau K\binom{n}{2}$
we consider the corresponding container vector $\ccC(\ccS)=(\cC(\bcS_1),\dots,\cC(\bcS_r))$
given by the Container Lemma. We denote by~$D(\ccS)$ its complement in~$E(K_n)$ given by
\[
	D(\ccS)=E(K_n)\setminus\big(\cC(\bcS_1)\cup\dots\cup\cC(\bcS_r)\big)\,.
\]
We observe that for any $B\in\cB$ the following two properties hold:
\begin{enumerate}[label=\alabel]
\item\label{it:cla} $\bigcup_{i\in[r]}\bigcup_{\sigma\in[s]}\cS_{i,\sigma}^B\subseteq\big(E(B)\setminus D(\ccS^B)\big)$ and
\item\label{it:clb}  $\big|E(B) \cap D(\ccS^B)\big| \leq |T^B|\leq t$,
since 
\[
	E(B) \setminus T^B = I^B_1\dcup\dots\dcup I^B_r \subseteq \cC(\bcS^B_1)\cup\dots\cup\cC(\bcS^B_r)=E(K_n)\setminus D(\ccS^B)\,.
\] 
\end{enumerate}
From~\ref{it:cla} and~\ref{it:clb} we infer that
\begin{align}
	\pr\big(G(n,p)\in\cB\big) &\leq \sum_{\ccS=(\bcS_1,\dots,\bcS_r)}p^{\left|\bigcup_{i\in[r]}\bigcup_{\sigma\in[s]}\cS_{i,\sigma}\right|}
		\cdot\pr\Big( \big|E(G(n,p))\cap D(\ccS)\big|\leq t\Big) \nonumber\\
		&\leq \max_{\ccS}\pr\Big( \big|E(G(n,p))\cap D(\ccS)\big|\leq t\Big)
		\sum_{\ccS}p^{\left|\bigcup_{i\in[r]}\bigcup_{\sigma\in[s]}\cS_{i,\sigma}\right|}\label{eq:clgoal}
		\,,
\end{align}
where the sum and the maximum are taken over all $r$-tuples $\ccS=(\bcS_1,\dots,\bcS_r)$
of $s$-tuples $\bcS_i=(\cS_{i,1},\dots,\cS_{i,s})$ of sets $\cS_{i,\sigma}$ of size at most~$\tau K\binom{n}{2}$ for $i\in[r]$ and $\sigma\in[s]$.
We will use property~\ref{cl:ii} of the Container Lemma to bound the maximum probability in~\eqref{eq:clgoal}, 
while our choice of constants allow us to derive a sufficient bound for the sum in~\eqref{eq:clgoal}.

For the maximum probability we first observe that for every $\ccS=(\bcS_1,\dots,\bcS_r)$ we have 
\begin{equation}\label{eq:Ddense}
	|D(\ccS)|>\frac{1}{R^2}\binom{n}{2}\,.
\end{equation}
For the proof we use the fact that for any $(r+1)$-colouring of $E(K_n)$ either there are more than
\[
	\frac{1}{2\binom{R}{k}}\binom{n}{k}
\] 
monochromatic copies of $K_k$ in the first $r$ colours or there are more than $\frac{1}{R^2}\binom{n}{2}$
edges having the last colour (see, e.g.,~\cite{RRS}*{Proposition~8}).

In view of this fact,  we consider $\cC(\bcS_1)\cup\dots\cup\cC(\bcS_r)\cup D(\ccS)$
as an $(r+1)$-colouring of~$E(K_n)$.
Owing to property~\ref{cl:ii} of the Container Lemma (see~\eqref{eq:clp2}),
every $\cC(\bcS_i)$ contains at most~$\eps\binom{n}{k}$
copies of $K_k$ and, hence, there are at most 
\[
	r\cdot\eps\binom{n}{k}\overset{\eqref{eq:clconst}}{=}\frac{1}{2\binom{R}{k}}\binom{n}{k}
\]
monochromatic copies in the first $r$ colours. Therefore, the mentioned fact above yields~\eqref{eq:Ddense}. 
In particular, the choice of $t$ combined with~\eqref{eq:Ddense}
yields $t < p|D(\ccS)|/2$ and, consequently, Chernoff's inequality (see, e.g.,~\cite{JLR}*{Theorem~2.1})
asserts 
\begin{equation}\label{eq:clmaxP}
	\pr\Big( \big|E(G(n,p))\cap D(\ccS)\big|\leq t \Big) \leq \exp\bigg(-\frac{p}{8R^2}\binom{n}{2}\bigg)
\end{equation}
for every $\ccS$ considered here.  In particular,~\eqref{eq:clmaxP} bounds the maximum probability 
considered in the R-H-S of~\eqref{eq:clgoal} and below we turn to the sum in~\eqref{eq:clgoal}. 

Owing to $|\cS_{i,\sigma}|\leq \tau K\binom{n}{2}$ for every $i\in[r]$ and $\sigma\in[s]$
we have 
\[
	 \sum_{\ccS=(\bcS_1,\dots,\bcS_r)}p^{\left|\bigcup_{i\in[r]}\bigcup_{\sigma\in[s]}\cS_{i,\sigma}\right|}
	 \leq
	 \sum_{m=0}^{r\cdot  s\cdot \tau K\binom{n}{2}}\binom{\binom{n}{2}}{m}2^{rs m}p^m
	 \leq
	 \sum_{m=0}^{r\cdot  s\cdot \tau K\binom{n}{2}}\left(\frac{\eu\binom{n}{2}}{m}2^{rs} p\right)^m\,.
\]
Since the function $m\mapsto (\eu \binom{n}{2}2^{rs}p/m)^m$ is unimodal and attains its maximum value
for $m_0=\binom{n}{2}2^{rs}p\geq r s\tau K\binom{n}{2}$ (see~\eqref{eq:clCp}), we can bound the summands in R-H-S above 
by the last one and obtain
\begin{align}
	\sum_{\ccS=(\bcS_1,\dots,\bcS_r)}p^{\left|\bigcup_{i\in[r]}\bigcup_{\sigma\in[s]}\cS_{i,\sigma}\right|}
	&\overset{\phantom{\eqref{eq:clCp}}}{\leq}
	\left(r  s \tau K\binom{n}{2}+1\right)\cdot \left(\frac{\eu\binom{n}{2}}{r  s \tau K\binom{n}{2}}2^{rs} p\right)^{r  s \tau K\binom{n}{2}}\nonumber\\
	&\overset{\phantom{\eqref{eq:clCp}}}{\leq}
	n^2\cdot\left(\frac{2^{rs} \eu D_p}{r sKD_{\tau}}\right)^{r  s \tau K \binom{n}{2}}\nonumber\\
	&\overset{\eqref{eq:clCp}}{=}
	n^2\cdot\big(2^{rs}\eu\cdot 50R^2rs\log(50R^2r)\big)^{r  s \tau K\binom{n}{2}}\nonumber\\
	&\overset{\phantom{\eqref{eq:clCp}}}{\leq}
	n^2\cdot\exp\left(r  s \tau K\binom{n}{2}\Big(rs+1+\ln\big(50R^2rs\log(50R^2r)\big)\Big)\right)\nonumber\\
	&\overset{{\eqref{eq:clCp}}}{\leq}
	n^2\cdot\exp\bigg(\frac{p}{16R^2}\binom{n}{2}\bigg)\nonumber\\
	&\overset{\phantom{\eqref{eq:clCp}}}{\leq}
	\exp\bigg(\frac{p}{12R^2}\binom{n}{2}\bigg)\,.\label{eq:clccS}
\end{align}
Finally, combining~\eqref{eq:clmaxP} and~\eqref{eq:clccS} 
with~\eqref{eq:clgoal} leads to 
\[
	\pr\big(G(n,p)\in\cB\big)
	\leq 
	\exp\bigg(-\frac{p}{8R^2}\binom{n}{2}\bigg) \cdot \exp\bigg(\frac{p}{12R^2}\binom{n}{2}\bigg)
	=
	\exp\bigg(-\frac{p}{24R^2}\binom{n}{2}\bigg)
	<
	\frac{1}{2}
\]
and Claim~\ref{claim_cliques_2} follows.
\end{proof}

\section{Concluding Remarks}\label{sec_con_rem}
In view of Theorem~\ref{thm_cycles} we may consider the following function for given integers $r\geq2$ and $k\geq 3$
$$f_r(k)= \min \big\{ |V(H)|\colon \girth(H) = k \text{ and } H \rightarrow (C_k)_r \big\}.$$
Theorem~\ref{thm_cycles} established that~$f_r(k) \leq R^{10k^2} k^{15 k^3}$, where~$R=R(C_k;r)$. In view of the known upper bounds on~$R(C_k;r)$ for even and odd~$k$, this establishes the upper bounds stated below in Theorems~\ref{thm_cycles_even} and~\ref{thm_cycles_odd}. These two theorems also provide complementary lower bounds.

\begin{thm}\label{thm_cycles_even}
There exist positive constants~$c_1$ and~$c_2$ such that for all~$k \geq 2$ and~$r \geq 2$,
$$ \exp \big( c_1 k \log r \big) \leq f_r(2k)  \leq  \exp \big(c_2 ( k^2 \log r + k^3 \log k) \big).$$
\end{thm}

For fixed~$k\geq 2$ Theorem~\ref{thm_cycles_even} shows that  $f_r(2k)$ is polynomial in $r$.

\begin{proof}
We will first show that $f_r(2k)  \leq  \exp \big(c_2 ( k^2 \log r + k^3 \log k) \big)$. In~\cite{E3} it was announced and in~\cite{BS} it was proved that, for every integer~$k \geq 2$, there exists a constant~$c$ such that every graph on~$n$ vertices with at least~$c n^{1+1/k}$ edges contains a copy of the cycle~$C_{2k}$. This implies that, if~$n$ is such that~${n \choose 2} / r  \geq c n^{1+1/k}$, i.e.~$n \geq c r^{k/(k-1)}$, then every edge colouring of~$K_n$ with~$r$ colours will have a monochromatic cycle~$C_{2k}$. Hence
\begin{equation} \label{eq_even_upper}
R(C_{2k};r) \leq c r^{k / (k-1)}.
\end{equation}
The upper bound $f_r(2k)  \leq  \exp \big(c_2 ( k^2 \log r + k^3 \log k) \big)$ now follows from substituting~\eqref{eq_even_upper} into Theorem~\ref{thm_cycles}. 

 We now turn our attention towards the lower bound in Theorem~\ref{thm_cycles_even}.  For any~$k \geq 2$ and~$r \geq 2$ consider any graph~$H$ with~$\girth(H) = 2k$ and the property~$H \rightarrow (C_{2k})_r$. Let~$\widetilde{H} \subseteq H$ be an edge minimal subgraph such that~$\widetilde{H} \rightarrow (C_{2k})_r$. Clearly the minimum degree of~$\widetilde{H}$ must be at least~$r$ and $\widetilde{H}$ must have girth at least~$2k$. Since any graph with girth~$2k$ and minimum degree~$r$ must have at least~$2 \sum_{i=0}^{k-1} (r-1)^i \geq c r^{k-1}$ vertices 
  the lower bound for~$f_r(2k)$ follows.
\end{proof}

The following theorem establishes a similar results for the odd case.

\begin{thm}\label{thm_cycles_odd}
There exist positive constants~$c_1$ and~$c_2$ such that, for all~$k \geq 1$ and~$r \geq 2$,
$$\exp \big(c_1  k r \big) \leq f_r(2k+1) \leq  \exp \big(c_2 k^2 \big( r \log r+ k  \log k \big) \big).$$
\end{thm}
For fixed~$k\geq 2$ it follows that~$e^{\Omega(r)} \leq f_r(2k+1) \leq e^{ O(r \log r)}$.

\begin{proof}
We fist show that $f_r(2k+1) \leq  \exp \big(c_2 k^2 \big( r \log r+ k  \log k \big) \big)$. As established in~\cite{BE}, 
\begin{equation} \label{eq_odd_upper}
2^{r} k \leq R(C_{2k+1};r) \leq (r+2)!  \cdot (2k+1).
\end{equation}
The upper bound for~$f_r(2k+1)$ follows from substituting the upper bound in~\eqref{eq_odd_upper} into Theorem~\ref{thm_cycles}.

To establish that $f_r(2k+1) \geq \exp \big(c_1  k r \big)$ for any~$k \geq 1$ and~$r \geq 2$, as before we begin by considering any graph~$H$ with~$\girth(H) = 2k+1$ and the property~$H \rightarrow (C_{2k+1})_r$. Note that~$\chi(H) > 2^{r}$, since otherwise the edges of~$H$ could be decomposed into~$r$ bipartite graphs, resulting in an~$r$-colouring of~$E(H)$ with no monochromatic odd cycle. Moreover, since $\chi(H) > 2^{r}$, there must be a subgraph~$\widetilde{H} \subset H$ with minimum degree at least~$2^r$. Since~$\widetilde{H}$ has at least girth~$2k+1$ and minimum degree~$2^r$, the number of vertices in $\widetilde{H}$ must be at least~$1+ 2^r \sum_{i=1}^{k-1}(2^r-1)^i \geq 2^{c r k}$ vertices for some $c>0$.
\end{proof}

For three special cases of $k$, we are able to deduce better bounds for $f_r(2k)$ using well known extremal 
constructions of graphs with girth $6$, $8$, and $12$, respectively.

\begin{thm}\label{thm_special_cases}
	 We have $f_r(6) = O(r^6)$, $f_r(8) = O(r^{12})$, and $f_r(12) = O(r^{30})$.
\end{thm}
Before proving Theorem~\ref{thm_special_cases}, we first introduce some notation and state an observation upon which the proof if based. Let~$\textrm{ex}(n;C_k)$ denote the maximum number of edges in an~$n$ vertex graph that does not contain a cycle of length~$k$. Similarly, let~$ex(n;C_3,C_4,\dots,C_{k-1})$ denote the maximum number of edges in a graph with girth~$k$.
\begin{fact}\label{obs_conclusion}
If
$\textrm{ex}(n;C_3,C_4,\dots,C_{2k-1})  >  r\cdot \textrm{ex}(n;C_3,C_4,\dots,C_{2k}),$ 
then
$f_r(2k) \leq n.$
\end{fact}
Indeed, by definition of the extremal function there exists a graph~$G$ on~$n$ vertices with girth~$2k$ that has $ex(n;C_3,C_4,\dots,C_{2k-1})$ edges. Clearly, every 
$r$-colouring of~$G$ yields a monochromatic subgraph  with at least $\textrm{ex}(n;C_3,C_4,\dots,C_{2k-1}) / r > \textrm{ex}(n;C_3,C_4,\dots,C_{2k})$ edges, 
which must contain a monochromatic~$C_{2k}$ since the monochromatic subgraph still has girth at least~$2k$.

\begin{proof}[Proof of Theorem~\ref{thm_special_cases}]
To make use of this fact to prove Theorem~\ref{thm_special_cases}, we use the result of 
Erd{\H o}s and Simonovits from~\cite{ES} that for every positive integer~$k$, we have 
\[
	\textrm{ex}(n;C_3,C_4,\dots,C_{2k+1}) = O(n^{1+1/k})\,.
\] 
Since any graph contains a bipartite subgraph with half of its edges we have
\begin{align}\label{eq_ES_lower}
\textrm{ex}(n;C_3,C_4,C_5,C_6,\dots,C_{2k})&\leq \textrm{ex}(n;C_4,C_6,\dots,C_{2k})\nonumber\\
& \leq 2 \cdot \textrm{ex}(n;C_3,C_4,C_5,C_6,\dots,C_{2k+1}) = O(n^{1+1/k})\,.
\end{align}
Erd{\H o}s and Simonovits conjectured in~\cite{ES} that for every positive integer~$k \geq 2$,
\begin{equation}\label{eq_ES_upper}
\textrm{ex}(n;C_3,C_4,\dots,C_{2k-1}) = \Omega (n^{1+1/(k-1)}).
\end{equation}
This has been observed for~$k=3$ by Klein (see~\cites{E38}) and follows 
for $k=4$ by the work of Singleton~\cite{Si}, 
and for $k=6$ by the work of Benson~\cite{B}. 
For~$k \in \{3,4,6\}$, inequalities~\eqref{eq_ES_lower} and~\eqref{eq_ES_upper} give that
$$
\textrm{ex}(n;C_3,C_4,\dots,C_{2k-1}) = \Omega (n^{1+1/(k-1)}) > r \cdot O(n^{1+1/k})= r\cdot \textrm{ex}(n;C_3,C_4,\dots,C_{2k}),$$
holds, provided that
\[
n \geq \tilde{c}\, r^{k(k-1)}\,,
\]
for some sufficiently large constant~$\tilde{c}$.
Consequently, Fact~\ref{obs_conclusion} yields $f_r(2k) \leq n = \Omega(r^{k(k-1)})$ for $k\in\{3,4,6\}$ and the theorem follows.
\end{proof}
We remark that establishing~\eqref{eq_ES_upper} for all~$k$, implies~$f_r(2k) = O(r^{k(k-1)})$ for all~$k$ by the same argument.

\end{document}